\documentclass[12pt,a4paper,reqno]{amsart}
\usepackage{amsthm,amsfonts,amsmath,amssymb}
\usepackage{mathrsfs}
\usepackage[top=1.6in, bottom=1.3in, left=1.3in, right=1.3in]{geometry}
\usepackage[colorlinks,linkcolor=blue,citecolor=blue,urlcolor=blue,hypertexnames=true]{hyperref}
\usepackage{enumerate}

\usepackage{booktabs}

\author{Ayan Nath}
\address{Kaliabor College, Kuwaritol, Assam, India}
\email{ayannath7744@gmail.com}

\author{Abhishek Jha}
\address{Indraprastha Institute of Information Technology, New Delhi, India}
\email{abhishek20553@iiitd.ac.in}



\title[On quotients of values of Euler's function on factorials]
{On quotients of values of Euler's function on factorials}

\theoremstyle{theorem}
\newtheorem{theorem}{Theorem}[section]
\newtheorem{lemma}[theorem]{Lemma}
\newtheorem{corollary}[theorem]{Corollary}

\newtheorem{definition}[theorem]{Definition}
\newtheorem{proposition}[theorem]{Proposition}
\newtheorem{remark}[theorem]{Remark}
\newtheorem*{notations}{Notations}
\newtheorem{hypothesis}[theorem]{Hypothesis}

\newcommand{\T}{\mathcal T}

\newcommand{\Mod}[1]{\ (\mathrm{mod}\ #1)}
\newcommand{\logx}{\log x}

\newcommand{\NN}{\mathbb N}

\renewcommand{\O}{\mathcal O}
\renewcommand{\phi}{\varphi}

\usepackage{mathtools}
\DeclarePairedDelimiter{\ceil}{\lceil}{\rceil}
\DeclarePairedDelimiter{\floor}{\lfloor}{\rfloor}

\begin{document}

\maketitle

\begin{abstract}
Recently, there has been some interest in values of arithmetical functions on members of special sequences, such as Euler's totient function $\phi$ on factorials, linear recurrences, etc.
In this article,
we investigate, for given positive integers $a$ and $b$, the least positive integer $c=c(a,b)$ such that the quotient
$\phi(c!)/\phi(a!)\phi(b!)$ is an integer. We derive results on the limit of the ratio $c(a,b)/(a+b)$ as $a$ and $b$ tend to infinity.
Furthermore, we show that $c(a,b)>a+b$ for all pairs of positive integers $(a,b)$ with an exception of a set of density zero.
\end{abstract}

\section{Introduction}
In recent years, there has been some interest in values of arithmetical functions, especially Euler's totient function $\phi,$ on members of special sequences.
Baczkowski et al. in \cite{luca:ijnt} investigated
arithmetic functions and factorials, precisely, $\phi(n!), d(n!)$ and
$\sigma(n!),$ where $d$ is the divisor counting function and $\sigma$ is summatory functions for divisors.
In \cite{luca:jnt-sequence}, Luca and Shparlinski obtained asymptotic formulas for moments of certain arithmetic functions with linear recurrence sequences. 
Further, Luca in \cite{luca:fib} considered $\phi(F_n),$ where $F_n$ is the $n$th Fibonacci number.
In \cite{luca:jnt-catalan}, Luca and St\u{a}nic\u{a} investigated quotients of the form
$\phi(C_m)/\phi(C_n),$ where $C_n$ is the $n$th Catalan number.

 Luca and St\u{a}nic\u{a} in \cite{luca:phinomial} considered an analogue
of binomial coefficients constructed by the Euler's totient function $\phi$, which were previously proven to be 
integral by Edgar in \cite{edgar:integers}. The authors defined
the $\phi$-\textit{actorial} as $n!_{\phi}=\phi(1)\phi(2)\cdots\phi(n),$ and the
\textit{phinomial} coefficient as
$$\binom{a+b}{a}_{\phi}=\frac{(a+b)!_{\phi}}{a!_{\phi}b!_{\phi}}.$$ 
In a similar spirit, we consider the quotient 
$$\frac{\phi((a+b)!)}{\phi(a!)\phi(b!)}.$$
Now, it does not take long to see that the above quantity is not always an integer.
In fact, the quotient is not an integer for all pairs $(a,b)$ with an exception of a set of density zero, as implied by Theorem
\ref{theorem:density}, though there do exist arbitrarily 
large $a$ and $b$ such that the expression is an integer; indeed, it is an integer
for $(a,b)=(n,\phi(n!)-1)$
(see Proposition \ref{lemma:low-ratio} for more such pairs).

Many interesting divisiblities in number theory are of the form $f(a)f(b)\mid f(c)$ 
where $f:\NN\to\NN$ is a function.
A classical result of Erd\H{o}s \cite{erdos:aufgaben} states there is an absolute constant $c$ so that if $a!b!$ divides $n!$ then $a+b< n +c \log n,$ but for infinitely many values of 
$n$ and some positive constant $c$, it follows that $n!/a!b!$ is an integer with $a+b=\floor{n+c\log n}.$ 
In fact, it is true that $(2n)!/(n!\floor{n+c\log n}!)$ is an integer for all $n$ with an exception of a set of density zero (see \cite{erdos:density}).
This motivates us to consider the quotient
$$\frac{\phi(c!)}{\phi(a!)\phi(b!)}.$$ 
As $\phi(n!)$ divides $\phi(m!)$ whenever $m\ge
n,$ it makes sense to study the least positive integer $c$ such that $\phi(a!)\phi(b!)$ divides
$\phi(c!).$ Let us denote the least such $c$ for a pair $(a,b)$ as $c(a,b).$ 
To understand the behaviour of $c(a,b)$ we plot a graph between $a+b$ and
$c(a,b)$ for all pairs $(a,b)$ with $1\le a,b\le 100$.
With the help of a computer, we obtain the plot shown in Figure \ref{fig1}.
\begin{figure}[ht]
\centering
\hspace*{-0.7cm}
\includegraphics[scale=0.7]{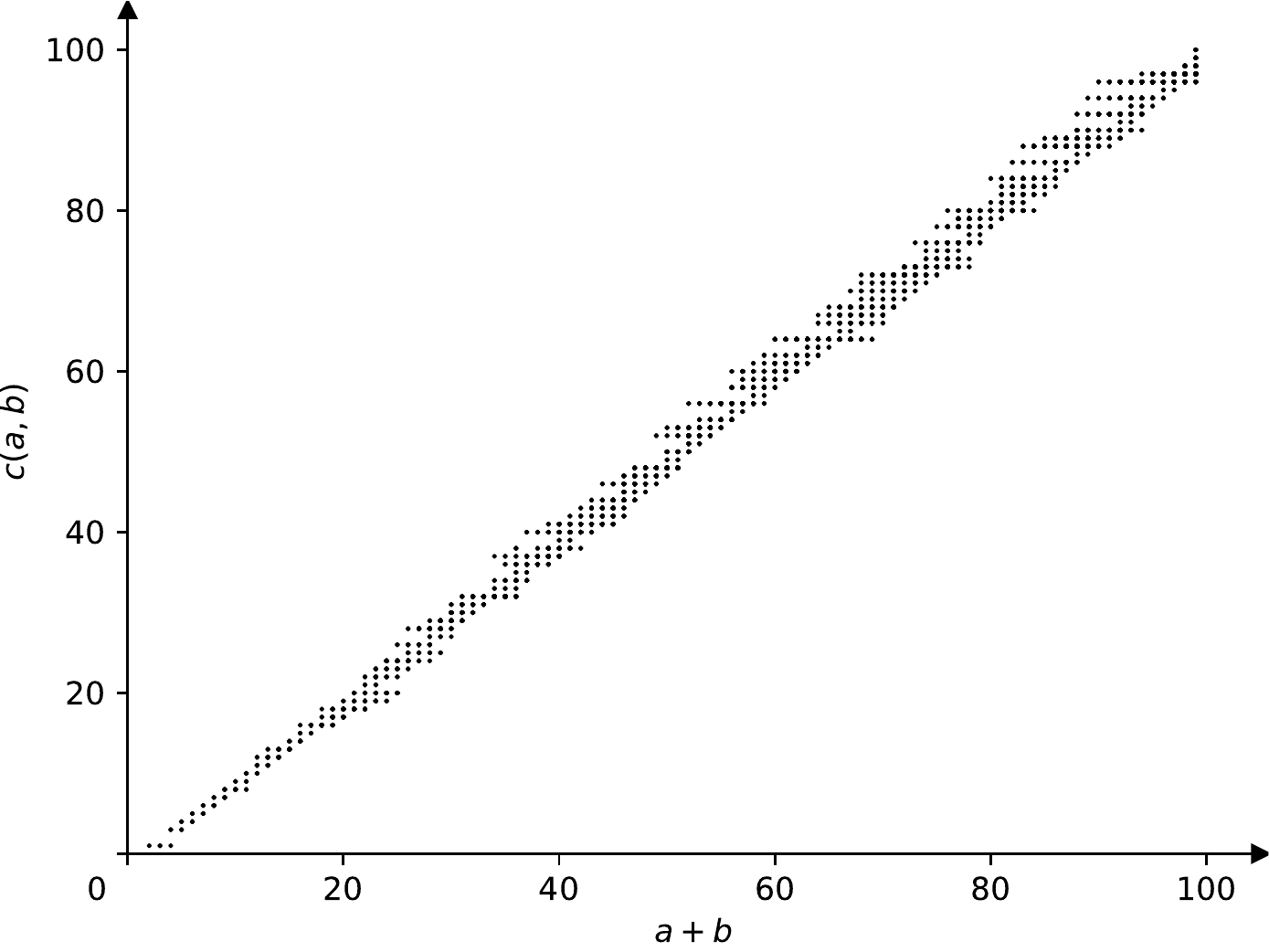}
\caption{Plot of $c(a,b)$ versus $a+b$.}\label{fig1}
\end{figure}

One immediately observes that the plot resembles the line $x=y.$ This suggests that $c(a,b)$ is close to $a+b,$ which motivates
us to study the ratio $r(a,b)=c(a,b)/(a+b)$ for pairs $(a,b).$

\begin{definition}
    Define $c(a,b)$ as the least positive integer $c$ such that $\phi(a!)\phi(b!)$ divides $\phi(c!),$ and denote the ratio $c(a,b)/(a+b)$ as $r(a,b).$
\end{definition}

It is observed using a computer that almost always $r(a,b)>1$.
Table \ref{table1} shows the proportion of pairs $(a,b)$ with $1\le a, b \le N$
such that $r(a,b)>1$ for various values of $N$.

\begin{table}[ht]
\caption{Proportions of pairs $(a,b)$ such that $r(a,b)>1.$}\label{table1}
\centering
\begin{tabular}{c|cccccccc}
\toprule
        $N$ & 100 & 200 & 300 & 400 & 500 & 600 & 700 & 800\\
        Proportion & 0.249 & 0.643 & 0.757 & 0.819 & 0.864 & 0.882 & 0.903 & 0.918\\
\bottomrule
\end{tabular}
\end{table}

The following theorem proved in this article confirms the evidence obtained in Table \ref{table1}.

\begin{theorem}\label{theorem:density}
    For all pairs of positive integers $(a,b),$ we have $r(a,b)>1$ with an exception of a set of density zero.
\end{theorem}

Seeing Figure \ref{fig1}, it is natural to ask whether $\lim_{a,b\to \infty}r(a,b)$ exists, and if yes, does it equal $1$? Or, if the limit does not exist, what are the values of the limit inferior and the limit superior? The following theorem shows the exact value of the limit inferior.

\begin{theorem}\label{theorem:liminf}
    $$\liminf_{a,b\to \infty}r(a,b)= 1.$$
\end{theorem}

Studying the limit superior is equivalent to obtaining bounds on the ratio $r(a,b).$
We prove the following ``sharp'' upper bound on $r(a,b).$

\begin{theorem}\label{theorem:limsup}
    For all large positive integers $a$ and $b$, we have $r(a,b)\le \frac 98.$
\end{theorem}
Unfortunately, it turns out that the sequence $r(a,b)$ fluctuates between
values; it never stabilizes. Figure \hyperref[fig2]{2} shows a plot of $n$ versus $r(n,n).$
\begin{figure}[ht]
\hspace*{-0.7cm}
\centering
\includegraphics[scale=0.7]{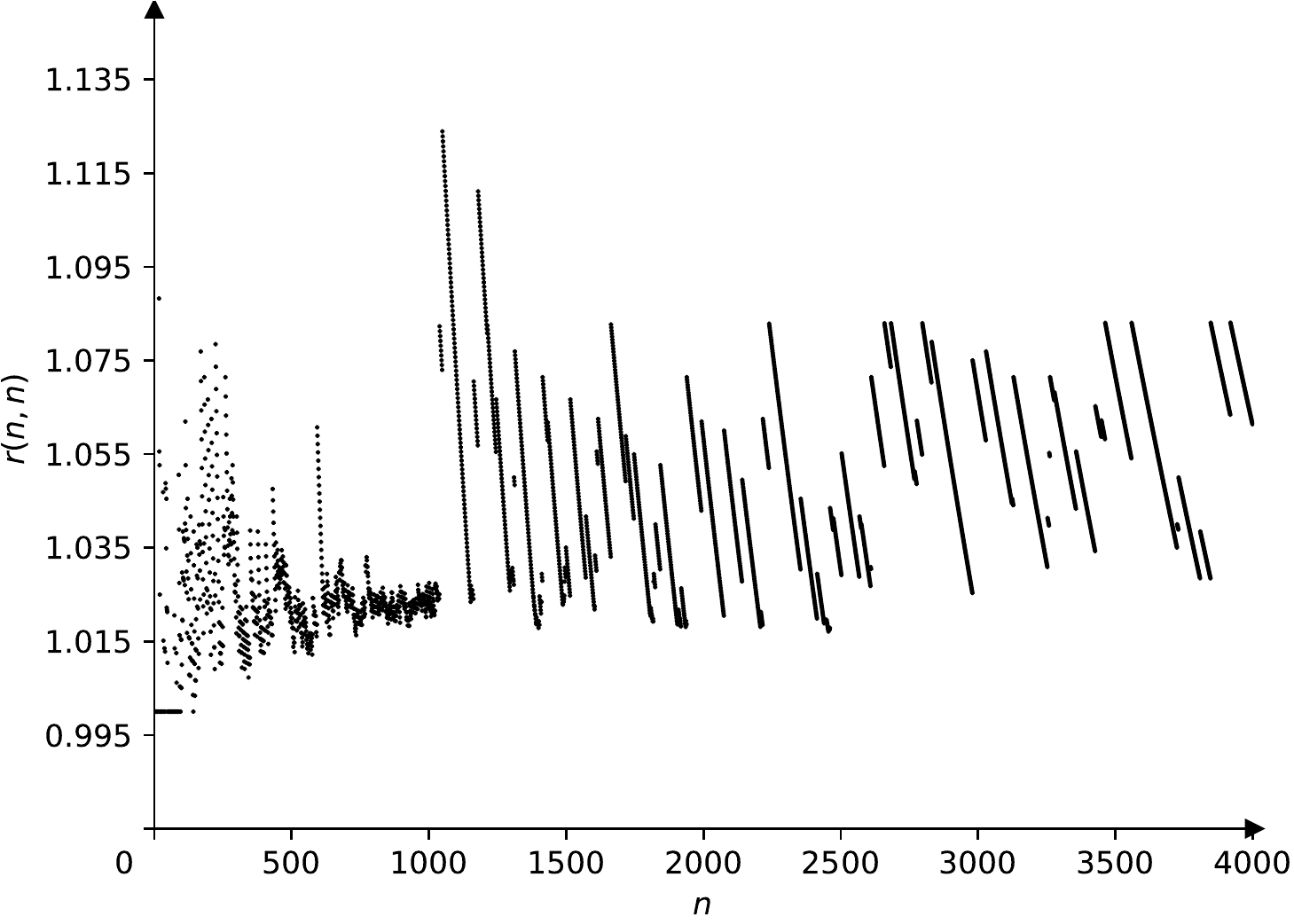}
\caption{Fluctuation of $r(n,n).$}\label{fig2}
\end{figure}

Under the hypothesis of Dickson's conjecture, a very well-believed and intuitive hypothesis in number theory, we prove the following theorem showing that the limit does not exist. 
\begin{theorem}\label{theorem:limsup-dickson}
    Assuming Dickson's conjecture, there are infinitely many positive integers $n$ such that $$r(n,n)\ge\frac{9}{8}-\frac{9}{8n}.$$
\end{theorem}
More importantly, Theorem \ref{theorem:limsup-dickson} shows that
Theorem \ref{theorem:limsup} is sharp in the sense that the constant $9/8$ cannot be improved any further.
It is now easy to see that Theorems \ref{theorem:limsup} and \ref{theorem:limsup-dickson} imply the following.
\begin{corollary}\label{corollary:limsup-dickson}
    Assuming Dickson's conjecture, we have
    $$\limsup_{a,b\to\infty}r(a,b) = \frac{9}{8}.$$
\end{corollary}

The paper is organized as follows. In Section \ref{section:prelim}, we state
some well-known results and prove a preliminary lemma. In Section
\ref{section:density}, we show that almost all pairs $(a,b)$ satisfy $r(a,b)>1$
and hence prove Theorem \ref{theorem:density}.
We present the proofs of Theorems \ref{theorem:liminf} and \ref{theorem:limsup}
in Sections \ref{section:liminf} and \ref{section:limsup}, respectively. Finally, in Section \ref{section:limsup-dickson}, we prove
Theorem \ref{theorem:limsup-dickson} under the hypothesis of Dickson's
conjecture which concludes the result of Corollary
\ref{corollary:limsup-dickson}.

\begin{notations}
\normalfont
We employ Landau-Bachmann notations $\O$ and $o$ as well as their associated
Vinogradov notations $\ll$ and $\gg$ with their usual meanings. Throughout the
article, the letters $p$ and $q$ are reserved for primes, and the letters $a$ and $b$ will always denote positive integers. As usual, define $\pi(x;m,a)$ to be the number of primes $p<x$ such that $p\equiv a\Mod m.$ For a prime
$p$ and a non-zero integer $n,$ define $\nu_p(n)$ as the exponent of $p$ in the prime factorisation of $n.$
For a non-zero rational number $r=a/b$ where $a$ and $b$ are integers, define $\nu_p(r)=\nu_p(a)-\nu_p(b)$ for any prime $p.$
\end{notations}

\section{Preliminaries}\label{section:prelim}
Here we list out some classical results which are going to be
helpful in our work. 
The following theorem is crucial in proving Lemma \ref{vqbound}.
\begin{theorem}[Siegel-Walfisz]\label{siegel}
    Let $C$ be a positive constant. If $a$ and $q$ are two relatively prime positive integers such that $a<q\ll (\log x)^C$, then
    $$ \left| \pi (x;q,a) - \frac{1}{\phi (q)} \cdot \frac{x}{\log{x}}  \right| \ll \frac{x}{(\log{x})^B}$$ 
    for some absolute constant $B>C+1.$
\end{theorem}
\begin{theorem}[Legendre]\label{legendre}
    For all positive integers $n$ and primes $p,$
    $$\nu_p(n!)=\sum_{i=1}^{\infty}\left\lfloor\frac{n}{p^i}\right\rfloor=\frac{n-s_p(n)}{p-1},$$
    where $s_p(n)$ is the sum of the digits of $n$ in base-$p.$
\end{theorem}

\begin{theorem}[Kummer]\label{kummer}
    Let $p$ be a prime. The highest power of a prime $p$ that divides $\binom{n}{k}$ is the number of carries in the addition $k+(n-k)$ in base-$p.$
\end{theorem}

We know that $\phi(n!)=n!\prod_{p\le n}\frac{p-1}{p}$. So, to calculate the exponent of $q$ in $\phi(n!)$, 
we are interested in estimating $\nu_q(\prod_{p<x}(p-1))$, which is exactly what the following lemma is about.

\begin{lemma}\label{vqbound}
    If $q$ is a fixed prime and $K$ is a real constant, then
    $$ \nu_{q} \left ( \prod_{p<x} (p-1) \right ) = \frac{q}{(q-1)^2} \cdot  \frac{x}{\log{x}} +\O \left (\frac{x}{\log^{K}x}  \right ).$$
\end{lemma}

\begin{proof}
    By a simple double counting argument, observe that
    $$\nu_{q} \left ( \prod_{p<x} (p-1) \right ) = \sum_{q^n<x} \pi (x;q^n,1).$$
    We split the summation into two intervals: $1\le q^n< (\log{x})^C$ and $(\log{x})^C\le q^n<x.$ So,
    $$\sum_{q^n<x} \pi (x;q^{n} ,1) = \sum_{q^n<(\log{x} )^{C} } \pi (x;q^{n} ,1)+\sum_{(\log )^C\le q^n < x} \pi (x;q^{n} ,1).$$
    Using Theorem \ref{siegel},
    \begin{align*}
        \sum_{q^n<(\log{x} )^{C} } \pi (x;q^{n} ,1) &= \sum_{q^n<(\log{x} )^{C}} \left(\frac{x}{\phi(q^n)\log x} + \O\left(\frac{x}{\log^B x}\right)\right) \\ 
                                                    &= \frac{x}{\log x}\sum_{q^n<(\logx)^C}\frac{1}{q^{n-1}(q-1)} + \O\left(\frac{x\log\logx}{\log^B x}\right)\\
                                                    &= \frac{q}{(q-1)^2}\cdot\frac{x}{\logx}(1+\O(\log^{-C}x))+\O\left(\frac{x\log\logx}{\log^B x}\right)\\
                                                    &= \frac{q}{(q-1)^2}\cdot\frac{x}{\logx} + \O\left(\frac{x\log\logx}{\log^B x}\right),
    \end{align*}
    where $B$ is the constant from Theorem \ref{siegel}.  Also,
    \begin{align*}
        \sum_{(\log x)^C\le q^n < x } \pi (x;q^{n} ,1) \le  \sum_{(\log x)^C\le q^n < x} \frac{x}{q^n} = \O\left(\logx\frac{x}{(\logx)^C}\right)=\O\left(\frac{x}{\log^{C-1}x}\right).
    \end{align*}
    Since $B>C+1,$ the lemma follows. 
\end{proof}

For convenience, let us define $\T(a,b;c)$ as follows.
\begin{definition}
    Define $\T(a,b;c)$ by $$\T(a,b;c)=\frac{\phi(c!)}{\phi(a!)\phi(b!)}.$$
\end{definition}

\section{Pairs \texorpdfstring{$(a,b)$}{(a,b)} such that \texorpdfstring{$r(a,b)>1$}{r(a,b)>1}}\label{section:density}

In this section, we will prove Theorem \ref{theorem:density}.
\begin{lemma}\label{lemma:density-helper}
    Let $A$ be any constant. If $a$ and $b$ are positive integers such that $b\ge a>\frac{b}{(\log b)^A},$ then $r(a,b)>1$ for all sufficiently large $b.$
\end{lemma}

\begin{proof}
    Set $c=a+b.$ We want to prove that $\T(a,b;c)$ is not an integer for all sufficiently large $b.$ By routine manipulations, we obtain that 
    $$\T(a,b;c)=\binom{a+b}{a}\prod_{p\le a}\frac{p}{p-1}\prod_{b<p\le c}\frac{p-1}{p}.$$
    We consider the largest power of $2$ dividing $\T(a,b;c)$ and claim that it is negative for all 
    large $b.$ Using Theorem
    \ref{kummer}, we know that $\nu_2\left(\binom{a+b}{a}\right)$ is the number
    of carries in the base-$2$ addition of $a$ and $b,$ which is at most
    $1+\log_{2}b.$
    Therefore, by Lemma \ref{vqbound}, we have
    \begin{align*}
        \nu_2(\T(a,b;c)) &= 2\left(\frac{c}{\log c}-\frac{a}{\log a}-\frac{b}{\log b}\right)+\O\left(\frac{b}{\log^{K}b}\right)\\
                       &\ll \frac{2b}{\log^{K-1}b} + 2\left(\frac{a+b}{\log (a+b)}-\frac{a}{\log a}-\frac{b}{\log b}\right)\\
                       &< 2\left(\frac{b}{\log^{K-1} b} +\frac{a}{\log (a+b)} - \frac{a}{\log a}\right).
    \end{align*}
    Since $b/a < \log^A b,$ the above expression is less than 
    $$2a\left(\frac{1}{\log^{K-1-A}b}+\frac{1}{\log(a+b)}-\frac{1}{\log a}\right).$$
    It is routine to check that $\frac{1}{\log(a+b)}-\frac{1}{\log a}$ is increasing in $a\in (0,\infty).$ Hence, as $a\le b,$ the above expression is bounded above by 
    $$2a\left(\frac{1}{\log^{K-1-A}b}+\frac{1}{\log b + \log 2}-\frac{1}{\log b}\right)< -\frac{\delta a}{\log^2 b}$$
    for some positive constant $\delta,$ provided $K-1-A>2,$ which can be ensured by taking $K$ large. And the proof is complete.
\end{proof}

With Lemma \ref{lemma:density-helper} in hand, it is easy to prove that the density of pairs $(a,b)$ such that $r(a,b)>1$ is $1.$ The proportion of pairs $(a,b)$ such that $a=b$ is $0$, so we can ignore them. It suffices to prove that the proportion of pairs $(a,b)$ with $a< b$ and $r(a,b)>1$ is $\tfrac 12$. Number of pairs $(a,b)$ such that $1\le a<b\le N$ and $r(a,b)>1$ is at least
\begin{align*}
    \sum_{b=N_0}^N\left((b-1)-\frac{b}{(\log b)^A}\right)&> \frac{N(N-1)}{2} - \frac{N^2}{(\log N)^A}+\O(1)\\
                                                         &=N^2\left(\frac 12 -\frac{1}{2N}-\frac{1}{(\log N)^A}+\O\left(\frac{1}{N^2}\right)\right),
\end{align*}
where $N_0$ is some constant only dependent on $a$ and $A.$
Thus, the density of such pairs is $\frac 12$, proving Theorem \ref{theorem:density}.
\section{The Limit Inferior}\label{section:liminf}
The proof of Theorem \ref{theorem:liminf} has two steps. First, we need to
prove that it is at least $1$, and then show that there are arbitrarily large $a$ and $b$ 
such that $r(a,b)\le 1.$ The following result implies that the
limit inferior is at least $1.$
\begin{lemma}\label{lemma:lower-bound}
    Let $A$ be any real constant. If $a$ and $b$ are positive integers with $a\le b$, then $c(a,b)\ge a+b-\frac{b}{\log^{A}b}$ for all sufficiently large $b.$
\end{lemma}

\begin{proof}
    Assume the contrary that $c=c(a,b)\le a+b-\frac{b}{\log^A b}$ is true infinitely often.
    We calculate $\nu_2$ of $\T(a,b;c)$ and claim that it is negative for all large $b$, which would contradict the assumption that $\T(a,b;c)$ is an integer. From Theorem \ref{legendre}, we have that $\nu_2(n!)=n+\O(\log n).$ Note that $b\le c\le a+b\le 2b.$ By Lemma \ref{vqbound},
    \begin{alignat*}{2}
        \nu_2(\T(a,b;c)) &= &&\nu_2(c!)-\nu_2(a!)-\nu_2(b!)+1+\nu_2\left(\prod_{b<p\le c} (p-1)\right)
                      \\&{}&&-\nu_2\left(\prod_{p\le a} (p-1)\right)\\
                        &= &&{}c-a-b + 2\left(\frac{c}{\log c}-\frac{a}{\log a}-\frac{b}{\log b}\right)+\O\left(\frac{b}{\log^{K}b}\right)\\
    \end{alignat*}

\vspace{-0.5cm}

\noindent  for some constant $K$ where we choose $K>A.$ Note that
     \begin{align*}
         \frac{c}{\log c}-\frac{a}{\log a}-\frac{b}{\log b} &\le \frac{c}{\log b} - \frac{a}{\log a}-\frac{b}{\log b}\\
                                                            &= \frac{c-b}{\log b}-\frac{a}{\log a}\\
                                                            &\le \frac{c-b-a}{\log a} <0.
     \end{align*}
     Therefore, 
     $$\nu_2(\T(a,b;c))\le -\frac{b}{\log^A b} + \O\left(\frac{b}{\log^K b}\right).$$
     Since $K>A,$ the above quantity is negative for all sufficiently large $b,$ which is a contradiction to the supposition that $\T(a,b;c)$ is an integer. The proof is complete.
\end{proof}

By Lemma \ref{lemma:lower-bound}, we see that $$r(a,b)\ge 1 - \frac{b}{(a+b)\log^A b}\ge 1-\frac{1}{2\log^A b}$$
for all sufficiently large $b.$ It is now evident that 
\begin{align*}
\liminf_{a,b\to\infty}r(a,b) \ge 1.
\end{align*}

What remains now to be proven is that there exist arbitrarily large $a$ and $b$
such that $r(a,b)\le 1.$ Indeed, $\T(a,\phi(a!)-1;\phi(a!))=1$ for all $a\ge4$,
so, $r(n,\phi(n!)-1)<1$ for all $n\ge4.$ This completes the proof of Theorem \ref{theorem:liminf}.

However, we prove the following result, which says that, given $a$, many $b$ satisfy $r(a,b)\le 1.$ The proof also demonstrates a natural way to construct such pairs.

\begin{proposition}\label{lemma:low-ratio}
    Let $a$ be a fixed positive integer. A positive proportion of positive integers $b$ satisfy $r(a,b)\le 1.$
\end{proposition}

\begin{proof}
    We provide a way to construct such $b.$ At the end, it will be clear that we can ensure $b\ge a,$ so let us assume $b\ge a$ from now onwards. Recall that
    $$\T(a,b;a+b)=\binom{a+b}{a}\prod_{p\le a}\frac{p}{p-1}\prod_{b<p\le a+b}\frac{p-1}{p}.$$
Set $D=\prod_{p\le a}(p-1)$. Clearly, $\prod_{b<p\le a+b}p$ divides
    $\binom{a+b}{a}$, and $D$ is relatively prime to $\prod_{b<p\le a+b} p.$
    Therefore, it suffices to prove that there exist infinitely many $b$ such
    that $D$ divides $\binom{a+b}{a}.$ Let the prime factorisation of $D$ be
    $\prod_{i=1}^{m}q_i^{\alpha_i}.$ By Theorem \ref{kummer}, we want the
    addition of $a$ and $b$ in base-$q_i$ to have at least $\alpha_i$ carries for each
    $i=1,2,\ldots,m.$ Therefore, we choose $b$ such that $b\equiv -a
    \Mod{q_i^{\alpha_i}}$ for each $i=1,2,\ldots, m,$ and we take 
$$b=k\prod_{i=1}^m q_i^{\alpha_i}-a=k\prod_{p\le a}(p-1)-a$$
for any positive integer $k$ such that $b\ge a.$ Thus, the proportion of positive integers $b$ such that $r(a,b)\le 1$ is positive.
\end{proof}

\begin{remark}
    It can be noted that the size of $b$ obtained in the above proof is around $\exp{a(1+o(1))},$ and the proportion of such $b$ is at least $\exp a(-1+o(1)).$
\end{remark}

\section{Upper bound on \texorpdfstring{$r(a,b)$}{r(a,b)}}\label{section:limsup}
In this section, we prove Theorem \ref{theorem:limsup}. 
\subsection{Setup.} 
Set $c=a+b+ \floor*{\frac{a+b}{8}}.$ Without loss of generality, assume $a\le b.$
We wish to prove that $\T(a,b;c)$ is an integer for all large $a$ and $b.$
It is easy to see that  
$$\T(a,b;c)=\binom{a+b}{a}(a+b+1)\cdots(a+b+\floor*{\tfrac{a+b}{8}})\prod_{p\le a}\frac{p}{p-1} \prod_{b<p\le c}\frac{p-1}{p}.$$
For brevity, call the above expression $\T.$ We prove that $\T$ is an integer by
showing that $\nu_q(\T)\ge 0$ for all primes $q.$ It is clear that $\nu_q(\T)\ge
0$ for all $q>a.$ So, let us assume that $q\le a$ for the rest of the argument.
We split our proof into three cases: $q\le 7,~q\in [8,a^{1/2}]$ and $ q\in(a^{1/2}, a]$. Take $a> 64$ so that the intervals do not overlap.

\subsection{Bounding \texorpdfstring{$\nu_q(\T)$}{v\_q(T)} for \texorpdfstring{$q\le 7$}{q≥7}.}
This case is straightforward using Lemma \ref{vqbound}. We have
\begin{align*}
    \nu_q(\T) &\ge \nu_q((a+b+1)\cdots(a+b+\floor*{\tfrac{a+b}{8}}))-\nu_q\left(\prod_{p\le a}(p-1)\right)\\
             &\ge \frac 1q \floor*{\frac{a+b}{8}} + \O\left(\frac{a}{\log a}\right),
\end{align*}
which is positive for all large $b$ as $a\le b.$

\subsection{Bounding \texorpdfstring{$\nu_q(\T)$}{v\_q(T)} for \texorpdfstring{$q\in[8,a^{1/2}]$}{q∈[8,√a]}.}
In this case, the estimates of Lemma \ref{vqbound} do not work because
the error terms get large as $q$ gets big. Hence, we require an alternative
bound for $\nu_q\left(\prod_{p\le a}(p-1)\right),$ which is stated as follows.

\begin{lemma}\label{lemma:vqbound-alternative}
    If $q> 7,$ then 
    $$\nu_q\left(\prod_{p\le a}(p-1)\right)\le \frac{0.23a}{q-1}+\frac{7\log a}{\log q}.$$
\end{lemma}
To prove this, we use the following preliminary lemma.

\begin{lemma}\label{lemma:pie-bound}
    Let $d>7$ be a positive integer relatively prime to $3,5$ and $7.$ Then, the number of primes in $\{d+1,2d+1,3d+1,\ldots,nd+1\}$ is at most $0.46n+7.$
\end{lemma}
\begin{proof}
    The proof is just a routine application of Inclusion-Exclusion principle. The number of elements in the set divisible by some positive integer $k$ is either $\floor*{\frac{n}{k}}$ or $\ceil*{\frac{n}{k}}.$ Applying Eratosthenes' sieve with respect to primes $3,5$ and $7$, the number of primes in the set is at most
    \begin{align*}
        n-\floor*{\frac n3}-\floor*{\frac n5}-\floor*{\frac n7}+\ceil*{\frac n{15}}+\ceil*{\frac n{21}}+\ceil*{\frac n{35}}-\floor*{\frac{n}{105}}< 0.46n+7,
    \end{align*}
    as desired.
\end{proof}

\begin{proof}[Proof of Lemma \ref{lemma:vqbound-alternative}]
    Similar to the proof of Lemma \ref{vqbound}, by a double counting argument and using Lemma \ref{lemma:pie-bound},
    \begin{align*}
        \nu_q\left(\prod_{p\le a}(p-1)\right) &= \sum_{2q^m<a}\pi(a;2q^m,1)\\
                                              &\le\sum_{2q^m<a}\left(\frac{0.46a}{2q^m}+7\right) \\
                                                &< \frac{0.23a}{q-1}+\frac{7\log a}{\log q},
    \end{align*}
    as desired.
\end{proof}

By a standard application of Theorem \ref{legendre},
\begin{align*}
    \nu_q((a+b+1)\cdots(a+b+\floor*{\tfrac{a+b}{8}})) &\ge \nu_q(\floor*{\tfrac{a+b}{8}}!)\\
                                                      &\ge \frac 1{q-1} \floor*{\frac{a+b}{8}}-\frac{\log(a+b)}{\log q} + \O\left(\frac 1{\log q}\right)\\
                                                      & \ge \frac{a+b}{8(q-1)}-\frac{\log b}{\log q} + \O\left(\frac 1{\log q}\right).
\end{align*}
Therefore,
\begin{align*}
    \nu_q(\T) &\ge \frac{a+b}{8(q-1)}-\frac{\log b}{\log q} + \O\left(\frac 1{\log q}\right)- \left(\frac{0.23a}{q-1}+\frac{7\log a}{\log q}\right).
\end{align*}
Since $a\le b$ and $q\le a^{1/2},$ we have
\begin{align*}
    \nu_q(\T) &\ge \frac{0.02a}{q-1} + \frac{b-a}{8(q-1)} - \frac{8\log b}{\log q}+\O\left(\frac 1{\log q}\right)\\
            &\ge 0.02\left(a^{1/2} + \frac{b-a}{400(q-1)} - \delta \log b\right)
\end{align*}
for some positive constants $\delta.$
Note that the above expression is non-negative if $a>\delta^2 \log^2 b.$ Again,
if $a<\delta^2\log^2 b,$ we see that $\frac{b-a}{400(q-1)}$ is greater
than $\delta \log b$ for all sufficienlty large $b.$ Thus, $\nu_q(\T)\ge 0$ and
this case is complete.

\subsection{Bounding \texorpdfstring{$\nu_q(\T)$}{v\_q(T)} for \texorpdfstring{$q\in(a^{1/2},a]$}{q∈(√a,a]}.}
Estimates like Lemma \ref{vqbound} and Lemma \ref{lemma:vqbound-alternative}
are not suitable as the margin for slack is quite thin in this
case. A careful analysis involving sharp bounds is required.
Note that $$\nu_q\left(\prod_{p\le a} (p-1)\right) = \pi(a;2q,1).$$
So, we are interested in upper bounding $\pi(a;2q,1)$ sharply.
\begin{lemma}\label{lemma:sharp-pi-bound}
    Let $q>7$ be a prime. The number of primes in $\{2q+1,4q+1,\ldots,2qk+1\}$ is at most $\floor*{\tfrac k2}+1,$ where equality can occur if $k=4$ and $q=11.$
\end{lemma}
\begin{proof}
    By Lemma \ref{lemma:pie-bound}, the number of primes in the set is at most
    $0.46k+7,$ which is less than or equal to $\floor*{\frac k2} + 1$ for all
    $k\ge 174.$  
    We are going to use computational aids for the rest of $k.$
    Fix some $k\le 173.$ Consider the aforementioned set modulo $3\cdot 5\cdot 7=105.$
    Obviously, $q$ can only be congruent to residues relatively prime to $105.$
    We check all such residues one by one. A number in the set cannot be a prime if it is not relatively prime to $105.$ 
    This procedure can be automated with a naive PARI/GP \cite{pari/gp} script for all $k\le 173,$ completing the proof.
\end{proof}
Similar to the previous cases, note that 
$$\nu_q\left(\floor*{\frac{a+b}{8}}!\right)\ge \floor*{\frac 1q\floor*{\frac{a+b}{8}}} \ge\floor*{\frac{a}{4q}}$$
as $a\le b.$ And by Lemma \ref{lemma:sharp-pi-bound}, $$\pi(a;2q,1) \le \floor*{\frac 12\floor*{\frac{a}{2q}}}+1=\floor*{\frac{a}{4q}}+1,$$
where we use the basic fact that $\floor*{\tfrac{a}{4q}}=\floor*{\tfrac 12\floor*{\tfrac{a}{2q}}}.$
Thus,
\begin{align*}
    \nu_q(\T) &\ge \nu_q\left(\floor*{\frac{a+b}{8}}!\right)+\nu_q\left(\prod_{p\le a}\frac{p}{p-1}\right)\\
             &\ge \floor*{\frac{a}{4q}} + 1 - \left(\floor*{\frac{a}{4q}}+1\right)=0,
\end{align*}
and the proof of Theorem \ref{theorem:limsup} is complete.

\section{Does the limit exist?}\label{section:limsup-dickson}
Until now, we have proved results on the inferior and superior limits. A
question that arises naturally is whether the limit exists. Numerical evidence
suggests that the sequence $r(a,b)$ keeps on fluctuating (see Figure \hyperref[fig2]{2}).

In this section, we provide a convincing argument for the non-convergence of
the sequence using Dickson's conjecture and prove Theorem \ref{theorem:limsup-dickson}.

\begin{hypothesis}[Dickson's Conjecture]\label{dickson}
    Let $a_1,a_2,\ldots,a_k$ be $k$ integers and $b_1,b_2,\ldots,b_k$ be $k$ positive integers. There are infinitely many positive integers $n$ for which all the elements in the set $\{a_1 + b_1n, a_2 + b_2n, \ldots, a_k + b_kn\}$ are primes, unless there is a congruence condition preventing this, i.e., there doesn't exist a prime $q$ (necessarily at most $k$) such that the product $(a_1+b_1n)(a_2+b_2n)\cdots (a_k+b_k n)$ is divisible by $q$ for all $n\in\{0,1,\ldots, q-1\}.$
\end{hypothesis}

Let $q>17$ be a sufficiently large prime such that $2q+1,~6q+1$ and $8q+1$ are primes but $iq+1$
is not prime for each $i\in\{10,12,14,16,18\}$.
Existence of infinitely many such primes $q$ can be proven under Dickson's conjecture.

\begin{lemma}\label{lemma:dickson}
    Assuming Dickson's conjecture, there exist infinitely many primes $q$ such
    that $2q+1,6q+1$ and $8q+1$ are primes but $iq+1$ is not prime for each
    $i=10,12,14,16,18.$
\end{lemma}
\begin{proof}
    The main idea is to find one such prime $q$ and apply Chinese Remainder Theorem.
    It is easy to see that $q$ must be greater than $7.$ Observe that $q\equiv 2 \Mod 3$ and $q\equiv 1 \Mod 5.$ Therefore, $iq+1$ is divisible by $3$ for each $i\in\{10,16\},$ and $14q+1$ is divisible by $5.$
    It is easily verified that $q=131$ satisfies the conditions of the statement. Indeed, 
    \begin{align*}
        2\cdot 131 + 1 = 263,~ 6\cdot 131+ 1= 787,~ 8\cdot 131+ 1 =1049,\\
        12\cdot 131+1=11^2\cdot 13,\qquad 18\cdot 131+ 1 =7\cdot 337.
    \end{align*}
Consider the system of congruences
\begin{align*}
    q \equiv 131 \Mod{11},\quad
    q \equiv 131 \Mod 7.
\end{align*}
This is solved by $q\equiv 54 \Mod{77}.$ So, if $q$ is of the form $77\ell
+54$, then $12q+1$ and $18q+1$ are not primes. By Dickson's conjecture, there
exist infinitely many $\ell$ such that $77\ell +54,2(77\ell +54)+1, 6(77\ell
+54)+1$ and $8(77\ell +54)+1$ are primes. No modular conditions are
preventing this because of the way we choose $q.$ The proof is complete.
\end{proof}

Take $n=8q+1$ and let
$c(n,n)=2n+m$ for some integer $m.$ By Lemma \ref{lemma:density-helper}, it is clear that $m$ is positive. Let us assume for now that $m < \frac{n-9}{4}$; it will be clear later why we are assuming this.
Recall that 
$$\T(n,n;2n+m)=\binom{2n}{n}(2n+1)(2n+2)\cdots (2n+m)\prod_{p\le n}\frac{p}{p-1}\prod_{n< p\le 2n+m}\frac{p-1}{p}.$$
We consider $\nu_q$ of the above expression. 
Since $q>17,$ and the base-$q$ representation of $n$ is $81$, hence there are no carries
in the addition $n+n$ in base-$q$. Therefore, by Theorem \ref{kummer}, $q$ does
not divide $\binom{2n}{n}.$ 
Observe that
$$\nu_q\left(\prod_{p\le n}\frac{p}{p-1}\right)\le -2,$$
and, whether $q$ divides
$\prod_{n<p\le 2n+m}(p-1)$ or not depends on how many of $9q+1,10q+1,\ldots,18q+1$
are primes. Clearly, $9q+1,11q+1,13q+1,15q+1$ and $17q+1$ are even.
The remaining $iq+1$ for each $i\in\{10,12,14,16,18\}$ are not primes by assumption.
Therefore, the product $\prod_{n<p\le 2n+m}(p-1)$ has no contribution to the exponent of $q$.
Lastly, the set $\{2n+1,2n+2,\ldots,2n+m\}$ is equivalent to $\{3,4,\ldots,m+2\}$ when considered modulo $q.$
Note that $2n+(q-2)= 17q$ is not divisible by $q^2.$
Therefore, for $\T(n,n;2n+m)$ to be an integer, we must have $m+2\ge 2q,$ which implies that $m \ge \frac{n-9}{4}.$ This is a contradiction to our initial assumption. Concluding,
$$c(n,n)\ge 2n+\frac{n-9}{4}=\frac{9n}{4}-\frac{9}{4}.$$
Thus, 
$$r(n,n) \ge \frac{9}{8} -\frac{9}{8n},$$
completing the proof of Theorem \ref{theorem:limsup-dickson}.
\begin{remark}
    By a quantitative version of Dickson's conjecture, known in its very general form as Bateman-Horn conjecture, the number of integers in $[1,x]$ which satisfy the condition of Theorem \ref{theorem:limsup-dickson} is of order at least $x/(\log x)^4$.
\end{remark}

\end{document}